\documentclass[10pt,reqno,oneside]{amsproc}
\title[The analytic construction method]{The analytic method of constructing local-in-time solutions of the incompressible Euler equations in Sobolev spaces}
\author[I.~Kukavica]{Igor Kukavica}
\address{University of Southern California, Los Angeles, CA 90089}
\email{kukavica@usc.edu}
\author[W.~S.~O\.za\'nski]{Wojciech  O\.za\'nski}
\address{Florida State University, Tallahassee, FL 32306, and Princeton University, Princeton, NJ 08540}
\email{wozanski@fsu.edu}
  \chardef\forshowkeys=0
  \chardef\refcheck=0
  \chardef\showllabel=0
  \chardef\sketches=0
\usepackage{enumitem}
\usepackage{datetime}
\usepackage{fancyhdr}
\usepackage{comment}
\allowdisplaybreaks

\ifnum\forshowkeys=1

  \usepackage[notref,notcite,color]{showkeys}
\usepackage{showkeys}
\fi
\usepackage[margin=1in]{geometry}
\usepackage{amsmath, amsthm, amssymb, mathtools, caption}
\usepackage{times}
\usepackage{graphicx}
\usepackage[usenames,dvipsnames,svgnames,table]{xcolor}
\usepackage{marginnote}
\usepackage[unicode,breaklinks=true,colorlinks=true,linkcolor=black,urlcolor=black,citecolor=black]{hyperref}
\usepackage[most]{tcolorbox}
\usepackage{tikz}

\ifnum\refcheck=1
  \usepackage{refcheck}
\fi
\begin{document}
\def\YY{X}
\def\na{\nabla }
\def\oo{{\Omega'}}
\def\ooo{\overline{\Omega'}}
\def\wo{{\Omega}}
\def\owo{\overline{\Omega}}
\def\OO{\mathcal O}
\def\SS{\mathbb S}
\def\CC{\mathbb C}
\def\RR{\mathbb R}
\def\TT{\mathbb T}
\newcommand{\Red}[1]{\begingroup\color{red} #1\endgroup} 
\def\T{\mathcal{T}}
\def\ZZ{\mathbb Z}
\def\HH{\mathbb H}
\def\RSZ{\mathcal R}
\def\LL{\mathcal L}
\def\SL{\LL^1}
\def\ZL{\LL^\infty}
\def\GG{\mathcal G}
\def\tt{\langle t\rangle}
\def\erf{\mathrm{Erf}}
\def\mgt#1{\textcolor{magenta}{#1}}
\def\ff{\rho}
\def\gg{G}
\def\sqrtnu{\sqrt{\nu}}
\def\ww{w}
\def\ft#1{#1_\xi}
\def\lec{\lesssim}
\def\gec{\gtrsim}
\renewcommand*{\Re}{\ensuremath{\mathrm{{\mathbb R}e\,}}}
\renewcommand*{\Im}{\ensuremath{\mathrm{{\mathbb I}m\,}}}

\ifnum\showllabel=0
 \def\llabel#1{\marginnote{\color{lightgray}\rm\small(#1)}[-0.0cm]\notag}
\else
 \def\llabel#1{\notag}
\fi

\newcommand{\eqnb}{\begin{equation}}
\newcommand{\eqne}{\end{equation}}

\newcommand{\nn}{\mathsf{n}}
\newcommand{\norm}[1]{\left\|#1\right\|}
\newcommand{\nnorm}[1]{\lVert #1\rVert}
\newcommand{\abs}[1]{\left|#1\right|}
\newcommand{\NORM}[1]{|\!|\!| #1|\!|\!|}
\newtheorem{theorem}{Theorem}
\newtheorem{Theorem}{Theorem}
\newtheorem{corollary}[theorem]{Corollary}
\newtheorem{Corollary}[theorem]{Corollary}
\newtheorem{proposition}[theorem]{Proposition}
\newtheorem{Proposition}[theorem]{Proposition}
\newtheorem{Lemma}[theorem]{Lemma}
\newtheorem{lemma}[theorem]{Lemma}
\theoremstyle{definition}
\newtheorem{definition}{Definition}[section]
\newtheorem{Remark}[theorem]{Remark}
\def\theequation{\thesection.\arabic{equation}}
\numberwithin{equation}{section}
\definecolor{mygray}{rgb}{.6,.6,.6}
\definecolor{myblue}{rgb}{9, 0, 1}
\definecolor{colorforkeys}{rgb}{1.0,0.0,0.0}
\newlength\mytemplen
\newsavebox\mytempbox
\makeatletter
\newcommand\mybluebox{%
    \@ifnextchar[
       {\@mybluebox}%
       {\@mybluebox[0pt]}}
\def\@mybluebox[#1]{%
    \@ifnextchar[
       {\@@mybluebox[#1]}%
       {\@@mybluebox[#1][0pt]}}
\def\@@mybluebox[#1][#2]#3{
    \sbox\mytempbox{#3}%
    \mytemplen\ht\mytempbox
    \advance\mytemplen #1\relax
    \ht\mytempbox\mytemplen
    \mytemplen\dp\mytempbox
    \advance\mytemplen #2\relax
    \dp\mytempbox\mytemplen
    \colorbox{myblue}{\hspace{1em}\usebox{\mytempbox}\hspace{1em}}}
\makeatother
\def\bnew{\colr }
\def\enew{\colb {}}
\def\bold{\colu }
\def\eold{\colb {}}
\def\phiij{\phi_{ij}}
\def\un{u^{(n)}}
\def\Bn{B^{(n)}}
\def\unp{u^{(n+1)}}
\def\unm{u^{(n-1)}}
\def\Bnp{B^{(n+1)}}
\def\Bnm{B^{(n-1)}}
\def\pn{p^{(n)}}
\def\pnm{p^{(n-1)}}

\def\ee{\mathrm{e}}
\def\eeo{\overline{\epsilon}}
\def\eet{\bar\epsilon}
\def\us{U}
\def\rr{r}
\def\weaks{\text{\,\,\,\,\,\,weakly-* in }}
\def\inn{\text{\,\,\,\,\,\,in }}
\def\cof{\mathop{\rm cof\,}\nolimits}
\def\Dn{\frac{\partial}{\partial N}}
\def\Dnn#1{\frac{\partial #1}{\partial N}}
\def\tdb{\tilde{b}}
\def\tda{b}
\def\qqq{u}
\def\lat{\Delta_2}
\def\biglinem{\vskip0.5truecm\par==========================\par\vskip0.5truecm}
\def\inon#1{\hbox{\ \ \ \ \ \ \ }\hbox{#1}}                
\def\onon#1{\inon{on~$#1$}}
\def\inin#1{\inon{in~$#1$}}
\def\FF{F}
\def\andand{\text{\indeq and\indeq}}
\def\ww{w(y)}
\def\startnewsection#1#2{\newpage \section{#1}\label{#2}\setcounter{equation}{0}}   
\def\nnewpage{ }
\def\sgn{\mathop{\rm sgn\,}\nolimits}    
\def\Tr{\mathop{\rm Tr}\nolimits}    
\def\div{\mathop{\rm div}\nolimits}
\def\curl{\mathop{\rm curl}\nolimits}
\def\dist{\mathop{\rm dist}\nolimits}  
\def\supp{\mathop{\rm supp}\nolimits}
\def\indeq{\quad{}}           
\def\period{.}                       
\def\semicolon{\,;}                  
\def\colr{\color{red}}
\def\colrr{\color{black}}
\def\colb{\color{black}}
\def\coly{\color{lightgray}}
\definecolor{colorgggg}{rgb}{0.1,0.5,0.3}
\definecolor{colorllll}{rgb}{0.0,0.7,0.0}
\definecolor{colorhhhh}{rgb}{0.3,0.75,0.4}
\definecolor{colorpppp}{rgb}{0.7,0.0,0.2}
\definecolor{coloroooo}{rgb}{0.45,0.0,0.0}
\definecolor{colorqqqq}{rgb}{0.1,0.7,0}
\def\colg{\color{colorgggg}}
\def\collg{\color{colorllll}}
\def\coleo{\color{colorpppp}}
\def\cole{\color{black}}
\def\colu{\color{blue}}
\def\colc{\color{colorhhhh}}
\def\colW{\colb}   
\definecolor{coloraaaa}{rgb}{0.6,0.6,0.6}
\def\colw{\color{coloraaaa}}
\def\comma{ {\rm ,\qquad{}} }            
\def\commaone{ {\rm ,\quad{}} }          
\def\cmi#1{{\color{red}IK: \hbox{\bf ~#1~}}}
\def\nts#1{{\color{red}\hbox{\bf ~#1~}}} 
\def\ntsf#1{\footnote{\color{colorgggg}\hbox{#1}}} 
\def\blackdot{{\color{red}{\hskip-.0truecm\rule[-1mm]{4mm}{4mm}\hskip.2truecm}}\hskip-.3truecm}
\def\bluedot{{\color{blue}{\hskip-.0truecm\rule[-1mm]{4mm}{4mm}\hskip.2truecm}}\hskip-.3truecm}
\def\purpledot{{\color{colorpppp}{\hskip-.0truecm\rule[-1mm]{4mm}{4mm}\hskip.2truecm}}\hskip-.3truecm}
\def\greendot{{\color{colorgggg}{\hskip-.0truecm\rule[-1mm]{4mm}{4mm}\hskip.2truecm}}\hskip-.3truecm}
\def\cyandot{{\color{cyan}{\hskip-.0truecm\rule[-1mm]{4mm}{4mm}\hskip.2truecm}}\hskip-.3truecm}
\def\reddot{{\color{red}{\hskip-.0truecm\rule[-1mm]{4mm}{4mm}\hskip.2truecm}}\hskip-.3truecm}
\def\tdot{{\color{green}{\hskip-.0truecm\rule[-.5mm]{6mm}{3mm}\hskip.2truecm}}\hskip-.1truecm}
\def\gdot{\greendot}
\def\bdot{\bluedot}
\def\pdot{\purpledot}
\def\ydot{\cyandot}
\def\rdot{\cyandot}
\def\fractext#1#2{{#1}/{#2}}
\def\ii{\hat\imath}
\def\fei#1{\textcolor{blue}{#1}}
\def\vlad#1{\textcolor{cyan}{#1}}
\def\igor#1{\text{{\textcolor{colorqqqq}{#1}}}}
\def\igorf#1{\footnote{\text{{\textcolor{colorqqqq}{#1}}}}}
\def\Sl{S_{\text l}}
\def\Sll{\bar S_{\text l}}
\def\Sh{S_{\text h}}
\newcommand{\p}{\partial}
\newcommand{\os}{{\overline{S}}}
\newcommand{\oos}{{\overline{\overline{S}}}}
\newcommand{\low}{\mathrm{l}}
\newcommand{\high}{\mathrm{h}}
\newcommand{\UE}{U^{\rm E}}
\newcommand{\PE}{P^{\rm E}}
\newcommand{\KP}{K_{\rm P}}
\newcommand{\uNS}{u^{\rm NS}}
\newcommand{\vNS}{v^{\rm NS}}
\newcommand{\pNS}{p^{\rm NS}}
\newcommand{\omegaNS}{\omega^{\rm NS}}
\newcommand{\uE}{u^{\rm E}}
\newcommand{\vE}{v^{\rm E}}
\newcommand{\pE}{p^{\rm E}}
\newcommand{\omegaE}{\omega^{\rm E}}
\newcommand{\ua}{u_{\rm   a}}
\newcommand{\va}{v_{\rm   a}}
\newcommand{\omegaa}{\omega_{\rm   a}}
\newcommand{\ue}{u_{\rm   e}}
\newcommand{\ve}{v_{\rm   e}}
\newcommand{\omegae}{\omega_{\rm e}}
\newcommand{\omegaeic}{\omega_{{\rm e}0}}
\newcommand{\ueic}{u_{{\rm   e}0}}
\newcommand{\veic}{v_{{\rm   e}0}}
\newcommand{\vp}{v^{\rm P}}
\newcommand{\tup}{{\tilde u}^{\rm P}}
\newcommand{\bvp}{{\bar v}^{\rm P}}
\newcommand{\omegap}{\omega^{\rm P}}
\newcommand{\tomegap}{\tilde \omega^{\rm P}}
\renewcommand{\vp}{v^{\rm P}}
\renewcommand{\omegap}{\Omega^{\rm P}}
\renewcommand{\tomegap}{\omega^{\rm P}}
\renewcommand{\d}{\mathrm{d}}
\newcommand{\R}{\mathbb{R}}
\newcommand{\C}{\mathbb{C}}
\newcommand{\N}{\mathbb{N}}
\renewcommand{\tt}{{\widetilde{\tau}}}
\newcommand{\im}{\mathrm{Im}}
\newcommand{\re}{\mathrm{Re}}
\newcommand{\id}{\mathrm{id}}

\newcommand{\tX}{{\widetilde{X}}}
\newcommand{\tY}{{\widetilde{Y}}}
\newcommand{\oY}{{\overline{Y}}}
\newcommand{\pmf}{{\psi^{(m)}_{\rm f}}}
\renewcommand{\pmb}{{\psi^{(m)}_{\rm b}}}

\begin{abstract}
We introduce a new method for constructing local-in-time solutions of the incompressible Euler equations in Sobolev spaces on an arbitrary Sobolev bounded domain. The method is based on  a  construction of an analytic solution in an analytically approximated domain, after which we apply analytic persistence to extend the analytic solution using given a priori bounds in Sobolev spaces. The method does not introduce \emph{any} modification or regularization of the equations themselves and appears applicable to many other PDEs.
\end{abstract}

\maketitle
\setcounter{tocdepth}{2}

\section{Introduction}
We let $r\geq 3$ be an integer, and we consider any bounded domain $\Omega\subset \R^3$ that is Sobolev in the sense that  $\Omega $ is piecewise a graph of a $H^{r+1}$ function, after a rotation of coordinates.    We consider
the three-dimensional incompressible Euler equations,
\eqnb\label{euler}
\begin{split}
\p_t u + u\cdot \nabla u + \nabla p &=0,\\
\div \, u&=0 \quad \text{ in }  \Omega,\\
u\cdot \nn  &=0 \quad \text{ on } \p\Omega,
\end{split}
\eqne
 where $\mathsf{n}$ denotes the unit outward normal vector to $\p \Omega$. We are  concerned with the classical problem of constructing local-in-time solutions in Sobolev spaces $H^r\equiv H^r (\Omega )$.

A well-known Sobolev a~priori estimate can be obtained by applying $\p^\alpha$ to the equation, where $|\alpha |=r$, multiplying by $\p^\alpha u$ and integrating in space, obtaining, formally,
\eqnb\label{sob_est}
\frac{\d }{\d t } \| u\|_{H^r} \leq C_0 \| u \|_{H^r}^2
,
\eqne
for some $C_0 = C_0 (\Omega , r) >0$; see~\cite{T1}. Since we are concerned with solutions continuous in time with values in $H^r$, we will consider \eqref{sob_est} in the integral sense,
\eqnb\label{sob_est1}
 \| u(t) \|_{H^r} \leq \| u(s) \|_{H^r} + C_0 \int_s^t \| u \|_{H^r}^2
\eqne
for $0\leq s<t$. We note that \eqref{sob_est1} gives  that $\| u(t) \|_{H^r} \leq 2 \| u_0\|_{H^r}$ for $t\in [0, T  ]$, where
\eqnb\label{time_lb}
 T   \coloneqq \frac{1}{8C_0 \| u_0 \|_{H^r}}, 
\eqne
and it suggests local-in-time well-posedness in $H^r$ until~$ T  $. That is, letting
\eqnb\label{V_def}
V \coloneqq \{ u\in H^r (\Omega ) \colon \div u =0\text{ in }\Omega , \, u\cdot \nn =0 \text{ on }\p \Omega  \},
\eqne
we have the following classical theorem.

\begin{theorem}\label{T1}
Given $u_0 \in V$, there exists a unique solution $u\in C([0, T  ];V)$ to the Euler equations \eqref{euler} such that $u(0)=u_0$ and \eqref{sob_est} holds (in the integral sense) on $[0, T  ]$. 
\end{theorem}

We note that the classical proofs of Theorem~\ref{T1} are tedious even in the case of $\R^3$
and require a modification of the equation and an appropriate convergence scheme.
The problems of constructing solutions become even more tedious when considering
free-boundary problems (see \cite{AKOT} for instance for a discussion on this topic).

The purpose of this note is to give a new proof of Theorem~\ref{T1}. Our approach is based on a Picard iteration using analytic spaces $X(\tau )$, $Y(\tau)$, where $\tau (t)$ denotes the radius of analyticity (see \eqref{EQ19c} for definitions). The use of such spaces is coupled with analytic approximation of both the initial data $u_0$ and the given Sobolev domain $\Omega$ by an analytic domain~$ \Omega^\varepsilon $. In that sense the data and the domain receive an \emph{analytic boost} (hence the name ``the analytic method''), which enables us to employ \emph{analytic dissipation} provided by the monotonicity of $\tau (t)$ to obtain a converging Picard estimate. We emphasize that the method does not introduce \emph{any} modification or regularization to the equations \eqref{euler}, and appears applicable to many PDEs of incompressible fluid mechanics.

 A clear example of the relevance of the method can be demonstrated in the case of the  inviscid MHD system, where the only available construction method is with the use of Elsasser \cite{El} variables, which fails for the MHD system  perturbed around a given background velocity field. The reason for this is that the use of these variables produces  an artificial pressure term in the equation for the magnetic field\footnote{ See~\cite{RW1} and the discussion following (1.5) in \cite{KO1} for details. In the case when background velocity field is not used, the method of Elsasser variables can still be used if the forcing in the $b$ equation is divergence-free, see \cite[Theorem~3.6]{R} for example.}. The analytic method introduced here  appears to be the only way of constructing solutions of the  homogeneous  MHD system posed on a bounded domain and perturbed around a given background velocity field. In particular, it helped resolve the 3D exact boundary controllability problem of the ideal MHD system, see \cite{KO1} for details;  see also \cite{C1,C2,KNV,R,RW2} for related results on the controlability of fluids.

\begin{Remark}[extendability of the method]
We suspect that the method presented here  can be easily generalized to other equations, such as the Boussinesq equation  and the Navier-Stokes equations.  Moreover, we conjecture that the method could be useful for  the free-boundary Euler equations, which were first resolved in the analytic setting by \cite{S} (see also \cite{E}), while later works, such as \cite{W1,W2,CS1,CS2} as well as \cite{AKOT,KO2}, focus on the Sobolev setting of the equations.

To be precise, the applicability of the method requires \emph{a priori} estimates in a sufficiently regular function space. While we verify these estimates in the Sobolev setting, other function spaces may also be appropriate. For example, for Besov spaces on $\Omega=\mathbb{R}^3$, the local \emph{a priori} bound follows from \cite{Z}. Considering, for example the Besov spaces which embed in $H^r$, one can obtain the persistence of analyticity estimate analogous to \eqref{002a} (see also Remark~\ref{rem3} for more details).

The second ingredient is that the equation  exhibits no more than one derivative loss, so that the analytic approach may overcome it, thereby yielding local existence for analytic data.  For example, the method cannot be used to construct solutions to the backward heat equation $\p_t u + \Delta u =0$. Indeed, if $u_0$ has a finite radius of analyticity, then there can be no solution, since otherwise this would contradict the fact that solutions to the forward heat equation are entire. In this example $\Delta u$ represents a loss of two derivatives, but the example can be generalized by replacing this operator by $(1+\Delta )^{s/2}$ for any $s> 1$ to obtain a similar example exhibiting loss of $s>1$ derivatives.
 
Finally, one needs the persistence of regularity, asserting that the solution can be continued in the analytic space as long as the Sobolev norm remains finite. Experience shows that when a~priori estimates in Sobolev spaces are available, the persistence property holds too (as demonstrated below \eqref{002a}); therefore, we believe that the last ingredient is not restrictive.

\end{Remark} 

\section{Proof of Theorem~\ref{T1}}

We proceed in a few steps.\\

\noindent\texttt{Step 0.} We introduce analytic spaces $X(\tau )$, $Y(\tau )$, for $\tau >0$.\\

We say that a domain $\Omega\subset \R^3$ is analytic if there exists a Lipschitz function $d\colon \R^3 \to \R$ such that  $  \Omega = \{ d > 0 \}$, $|\nabla d |$ is bounded away from $0$ on some neighbourhood of $\p\Omega$,  and
\eqnb\label{def_an_domain}
d \text{ is a real analytic function on }Q,
 \eqne
 where $Q$ is an open neighbourhood of~$\p \wo$. We will use the  Komatsu \cite{K} notation for derivatives,
\eqnb\label{komatsu_not}
\p^i \coloneqq \bigsqcup_{\alpha \in \{ 1,2,3\}^i} D^\alpha u,
\eqne
where 
\[
D^\alpha u \coloneqq \p_{\alpha_1} \ldots  \p_{\alpha_i} u.
\]
For instance, $\p^2$ is the $3\times 3$ matrix of second order derivatives, so that, in particular, the same mixed derivative is included in $\p^i$ multiple times. 
As a consequence of such notation, we obtain
\[
\| D^i u \| = \sum_{\alpha \in \{ 1,2,3\}^i} \| D^\alpha u \|.
\]
Moreover, we have the product rule 
\eqnb\label{komatsu_product}
\p^i (fg) - f \p^i g = \sum_{k=1}^i {i \choose k } \p^k f \p^{i-k} g ,
\eqne
where the last product denotes the tensor product of derivatives that is consistent with \eqref{komatsu_not}
(see (5.1.3) in~\cite{K}). The product rule \eqref{komatsu_product} is a consequence of the  identity
\[
\sum_{\substack{\alpha' \subset \alpha \\ |\alpha' |=k }} {\alpha \choose \alpha' } = {m \choose k}
\]
for each $m\in \N$, $k\in \{ 0 , \ldots , m \}$, and~$|\alpha |=m$. 
 
We say that a vector field $\T=\sum_{i=1}^n a_i \p_i$ is a \emph{tangential operator} to $\p \Omega$ if $\T$ is a global analytic vector field such that $\T  d  =0$ on $\p \Omega$, where $ d  (x) $ is the distance function to the boundary $\p \Omega$, taking positive values inside $\Omega$ and negative outside; see \cite[Section~2.1]{CKV}
and \cite{JKL2}
for extensive discussions on tangential operators. We will abuse the notation, and denote by $\T$  the tensor of a complete system of tangential derivatives, namely a system such that  the components of~$\T$ span the tangent space at $x$, for each $x\in \p \Omega $; see  \cite[Theorem~3.1]{JKL2} for details. The system is chosen depending on the domain $\Omega$ only. For instance, as shown in \cite{JKL2}, we can use one tangential derivative in~2D and three in~3D, with a constant number $K=n(n-1)/2\in \N$ in any dimension $n\geq 2$.
 
 We define analytic norms as
 \begin{align}
  \begin{split}
   &
   \Vert v \Vert_{X(\tau)}
   \coloneqq 
  \sum_{i+j\geq r } c_{i,j} 
    \Vert \p^i \T^j  v\Vert , \qquad \text{ where } c_{i,j} \coloneqq \frac{(i+j)^r}{(i+j)!}\tau^{i+j-r }
  \overline{\epsilon}^i \epsilon^j,
    \\&
   \Vert v \Vert_{\tX(\tau)}
   \coloneqq 
   \Vert v \Vert_{X(\tau)}
   + \Vert v \Vert_{H^{r}},
  \\&
   \Vert v \Vert_{Y(\tau)}
   \coloneqq 
   \sum_{ i+j\geq r+1}
    \frac{
     (i+j)^{r+1}
        }{
     (i+j)!
    }
    \tau^{i+j-r-1}
    \overline{\epsilon}^i \epsilon^{j}
    \Vert \p^i \T^j  v \Vert,
    \\&
   \Vert v \Vert_{\widetilde{Y}(\tau)}
   \coloneqq 
  \tau  \Vert v \Vert_{Y(\tau)}
   + \Vert v\Vert_{H^{r}} ,\\&
   \Vert v \Vert_{\oY (\tau)}
   \coloneqq 
  \Vert v \Vert_{Y(\tau)}
   + \Vert v \Vert_{H^{r}},
  \end{split}
   \label{EQ19c}
  \end{align}
where $0< \overline{\epsilon} \ll \epsilon \ll 1$ are constants determined 
by~\eqref{p_est} below
and all the norms are taken on~$\Omega$.  We note that the $\tX$, $\tY$, $\oY$ norms induce Banach subspaces on $V$.  We also use the notation
\[
\| \cdot \| \coloneqq \| \cdot \|_{L^2 (\wo )}
\]
for the $L^2$ norm.  It follows from \eqref{def_an_domain}  that 
\eqnb\label{n_is_ana}
 \nn \circ \gamma  \in \tX ( 1; \widetilde{\Omega }),
\eqne
provided $\overline{\epsilon}/ \epsilon$ and $\epsilon$ are sufficiently small (depending on $\Omega$), where  $\gamma (x) $ is the closest point on $\p \Omega$ to $x\in \widetilde{\Omega}$, and   $ \widetilde{\Omega}$ is some neighbourhood  of $\partial \Omega$, and the space $\tX ( \tau_0; \widetilde{\Omega })$ is defined by the norms in~\eqref{EQ19c} above, i.e., $\| f\|_{\tX (\tau ; \widetilde{\Omega })}\coloneqq \|  f\|_{X (\tau ; \widetilde{\Omega })}+ \| f \|_{H^r(\widetilde{\Omega })}$, where
\eqnb\label{def_XtauOmega}
\| f\|_{X(\tau;\widetilde{\Omega })} \coloneqq \sum_{i+j\geq r} c_{i,j} \| \p^i \T^j f \|_{L^2 (\widetilde{\Omega })}.
\eqne
 In order to see \eqref{n_is_ana}, it can be shown that $\mathsf{n}\circ Y = \nabla  d $ on a sufficiently small neighbourhood of $\p \Omega$ (see \eqref{whoisnad} for a proof), and one can show that the $X(1;\widetilde{\Omega })$ norm is finite, by commuting the tangential derivatives $\T$ in the definition \eqref{def_XtauOmega} by employing the commutator estimates (see \cite[Lemma~3.1]{KOS}, \cite[Lemma~3.5]{CKV} or \cite[Section~5]{K} for details) and by using the Cauchy estimates (i.e., analytic estimates on subdomains, see \cite[Lemma~4.4]{SC} for details). We also note that analyticity radius $1$ in \eqref{n_is_ana} is achieved by taking $\overline{\epsilon}/\epsilon$, $\epsilon$ sufficiently small, depending on $\Omega$.

We note that the definitions \eqref{EQ19c} give 
\eqnb\label{an1}
\| f \|_{\tX} \lec_r \| f \|_{\tY } \qquad \text{ and } \qquad 
\| \na f \|_{\tX} \lec_r \| f \|_{\oY }
\eqne
as well as
\eqnb\label{an3}
\sum_{i+j\geq r} c_{i,j} \| \p^{i+1} \T^j  f \|  \lec_r \| f \|_{Y }.
\eqne
Here, and throughout the paper, we used the shorthand notation ``$\lec_r$'' for ``$\leq C(r)$ for some $C(r)>0$''.\\

\noindent\texttt{Step 0a.} We recall some estimates in the analytic norm of terms arising from incompressible fluids.\\

First, denote by 
\eqnb\label{def_Salpha}
S_{ij} (v,w) \coloneqq \p^i \T^j ((v\cdot \nabla )w )-v \cdot \nabla \p^i \T^j w
\eqne  
the commutator term arising from passing derivatives inside the advecting term. According to~(3.15) in~\cite{KOS}, we have
\eqnb\label{product_curved}
\sum_{i+j\geq r }c_{i,j}   \| S_{i j} (u,v)\|  \lec_r \| v \|_{\tY (\tau )} \| u \|_{\tX (\tau )} +\| v \|_{\tX (\tau )} \| u \|_{\tY (\tau )}.
\eqne
Second, if $p$ is a pressure function satisfying
\eqnb
\begin{split}
 -\Delta p
   &=
   \partial_{i} f_j \partial_{j} g_i \quad \text{ in }\Omega,\\
   - \nabla p \cdot \nn
   &= (v\cdot \nabla w)\cdot \nn \quad \text{ on }\p \Omega,
\end{split}
\notag
\eqne
for some  vector fields $f,g$ and  divergence-free  $v,w$ , then
\eqnb\label{p_est}
\| \nabla p \|_{X} \lec_r \| f\|_{\widetilde{X}}\| g\|_{\widetilde{X}}+\| v\|_{\widetilde{X}}\| w\|_{\widetilde{X}},
\eqne
provided $\epsilon$ and $\overline{\epsilon}/\epsilon$ are sufficiently small, see Lemma~3.3 in \cite{KOS} for a proof.\\

\noindent\texttt{Step 1.} We construct an analytic solution on an analytic domain on the time interval~$[0, T  ]$.\\

In this step we suppose that $\Omega$ is analytic in the sense of \eqref{def_an_domain} and, given initial data $u_0\in V$ such that $u_0 \in X(\tau_0)$, we find an analyticity radius function $\tau \in C^1 ([0, T  ];(0,\infty))$ and a unique solution $u \in C^0 ([0, T  ];\widetilde{X}(\tau(t)) )$ to \eqref{euler} satisfying the \emph{analytic a priori estimate}
\eqnb\label{apriori_an}
\frac{\d }{\d t } \| u \|_{\tX } - \dot \tau \| u \|_{Y} \leq C_1 \| u \|_{\tX }\|  u \|_{\tY}
\eqne
for all $t\in [0, T  ]$, where $ T  $ is given in~\eqref{time_lb}.\\

We emphasize that, even though we aim to find an analytic solution $u$, the time $ T  $ of existence depends only on the Sobolev norm $\| u_0 \|_{H^r}$ of the initial data (recall~\eqref{time_lb}). This is possible due to the \emph{persistence of analyticity}; see \eqref{002a} below. \\

In order to obtain~\eqref{apriori_an}, we apply $\p^i \T^j $ to the Euler equation \eqref{euler}, multiply by $\p^i \T^j u$ and integrate over $\Omega$ to get 
  \begin{align}
  \begin{split}
   &
   \frac12
   \frac{\d}{\d t}
   \int
   |\partial^i \T^j u|^2
       =-
         \underbrace{\int  u \cdot \nabla  \partial^i \T^j u \cdot \partial^i \T^j u
   }_{=0}     - \int S_{ij} (u  ,u )\cdot \partial^i \T^j u - \int \p^i \T^j \nabla p \cdot \p^i \T^j u.
  \end{split}
   \label{EQ08}
  \end{align}

Thus, recalling \eqref{product_curved} and~\eqref{p_est}, we can multiply \eqref{EQ08} by $c_{i,j}$ and sum in $i+j \geq r$, and use \eqref{an1} to obtain the a~priori estimate~\eqref{apriori_an}. 

We can now choose $\tau (t)$ such that \eqref{apriori_an} lets us control the solution until the final time $ T  $ as in~\eqref{time_lb}. Namely, if $\tau$ satisfies
\eqnb\label{tau_choice}
-\dot \tau \geq 2 C_1 \tau \| u \|_{\tX}
,
\eqne
then \eqref{apriori_an} gives
\eqnb\label{002a}
\frac{\d }{\d t } \| u \|_{\tX } - \frac{\dot \tau}2  \| u \|_{Y } \leq C_1  \| u \|_{\tX }  \| u \|_{H^r} .
\eqne
This shows that $\| u \|_{X(\tau )}$ remains bounded and $-\dot \tau \| u \|_{Y(\tau )}$ remains integrable as long as $\| u (t) \|_{H^r}$ remains bounded, provided the analyticity radius $\tau (t)$ satisfies~\eqref{tau_choice} and stays positive. 

\begin{Remark}[The key of the analytic method]\label{rem3}
We emphasize that the above trick  is the essence of the method in the sense that the analytic dissipation is used to absorb the derivative loss in the a~priori bound  \eqref{apriori_an}. To be precise, the derivative loss  is encoded in the $\tY$ norm on the right-hand side of \eqref{apriori_an}. Indeed, since the power of $i+j$ the definition~\eqref{EQ19c} of $\| \cdot \|_{Y} $  is larger than in the definition of $\| \cdot \|_{X}$, we see that, at the level of spatial derivatives, the $\| u \|_{\tY}$ norm is, roughly speaking, at the same level of $\| \nabla u \|_{\tX}$. Thus, the appearance of ``$-\dot \tau \| u\|_{Y}$'' on the left-hand side of \eqref{apriori_an} lets us absorb the ``$Y$ portion'' of the $\tY$ norm of the right-hand side, leaving only the ``Sobolev portion'', thanks to the condition \eqref{tau_choice}. Thus the ``$-\dot \tau \| u\|_{Y}$'' term  provides a notion of a dissipation (the \emph{analytic dissipation}) even though the equation \eqref{euler} is inviscid. Since $\| u \|_{H^r}$ can be controlled independently (see below), we have thus reduced 
the quadratic estimate in the analytic norm \eqref{apriori_an} into \eqref{002a}, which is linear in the analytic norm $\tX$ and naturally leads to local well-posedness in the analytic spaces (i.e., \emph{persistence of analyticity}). 
\end{Remark}

In order to control $\| u \|_{H^r}$, the standard Sobolev estimate (see~\cite{T1}) gives 
\[
\frac{\d }{\d t} \| u \|_{H^r} \leq C_0  \| u \|_{H^r}^2,
\] 
which in particular implies that 
\eqnb\label{002b}
\| u (t) \|_{H^r} \leq  2 \| u_0 \|_{H^r} \qquad \text{ for }\quad 
 0\leq  t\leq \frac{1}{4C_0 \| u_0 \|_{H^r}}.
\eqne
For such $t$, the inequality \eqref{002a} gives $\frac{\d }{\d t} \| u \|_{\tX} \leq 2C_0 \| u_0 \|_{H^r} \| u \|_{\tX }$, so that
\eqnb\label{002c}
\| u(t) \|_{\tX } \leq \| u_0 \|_{ \tX } \ee^{2C_0 t \| u_0 \|_{H^r} }  =: G(t),
\eqne
provided \eqref{tau_choice} holds. We can thus fix $\tau (t)$ by an ODE of the form  \eqref{tau_choice}, except with $\| u \|_{\tX}$ replaced by $G$, i.e., we define $\tau (t)$ by the ODE 
\eqnb\label{tau_choice_actual}
\dot \tau (t) = - 2C_1 \tau (t) G(t) ,\qquad \tau (0) = \tau_0.
\eqne
With such a choice of $\tau$, the inequality~\eqref{002a} gives the a~priori estimate
\eqnb
\label{004a}
\frac{\d }{\d t } \| u \|_{\tX } - \frac{\dot \tau}2  \| u \|_{Y } \leq 2 C_1  \| u \|_{\tX }  \| u_0 \|_{H^r} 
\eqne
for $t\in [0, T  ]$.   Moreover, we emphasize that $\tau$ remains bounded below away from $0$ on $[0,T]$, as $G(t)$ remains bounded on $[0,T]$, due to \eqref{002b}.  A simple Picard iteration now lets one construct a solution $u \in C([0, T  ];\tX (\tau ))$ such that $-\dot \tau u \in L^1((0, T  );\tY (\tau ))$ and the upper bounds \eqref{002b}, \eqref{002c} hold.  For example, one can consider $u^{(0)} \coloneqq u_0$, and 
$u^{(n)} \coloneqq u_0 - \int_0^t \left( u^{(n-1)} \cdot \nabla u^{(n-1)} + \nabla p^{(n-1)}\right)$ for $n\geq 1$, where $p^{(n)}$, for $n\geq 0$, is a solution to the Poisson equation $\Delta p^{(n)} = - \p_j u^{(n)}_i \p_i u^{(n)}_j$  in $\Omega$ with Neumann boundary condition $\p_{\mathsf{n}} p^{(n)} = (( u^{(n)} \cdot \nabla ) u^{(n)})\cdot \mathsf{n}$ on $\p \Omega$. Then one can show that $u^{(n)}$ is Cauchy in $C ([0,T]; \tX (\tau ))$ and satisfies $\sup_{n\geq 0} \int_0^T \left( -\dot \tau \| u^{(n)} \|_{Y (\tau )} \right) <\infty$, which implies the existence of the desired solution $u$. We refer the reader to \cite[Section~2.2]{KOS} for the details of such a Picard scheme.\\

\noindent\texttt{Step 2.} We construct a Sobolev solution on an analytic domain.\\

Here we prove Theorem~\ref{T1} in the case when $\Omega$ is analytic in the sense of~\eqref{def_an_domain}. To this end, we first show the following.

\begin{lemma}[Analytic approximation lemma]\label{L_an_approx}
Let $\Omega$ be a bounded, analytic domain (in the sense of \eqref{def_an_domain}).  Provided $\overline{\epsilon}/ \epsilon$ and $\epsilon$ (in \eqref{EQ19c}) are chosen sufficiently small, then, for all sufficiently small $\varepsilon>0$ there exists $S_\varepsilon \in B(V , V )$ with $\| S_\varepsilon \| \lec 1$ such that,  for each $v\in V$  , $ S_{\varepsilon } v  \in \oY (1)$  and
\[
\| S_\varepsilon v - v \|_{H^r} \to 0
\]
as $\varepsilon \to 0$. 
\end{lemma}
\begin{proof}
We denote by $E\in B(H^r (\Omega ), H^r (\R^3))$ the Sobolev extension operator (see \cite[Theorem~5.22]{AF}). Given $v\in \mathcal{H}_r$, we set
\[
v_\varepsilon \coloneqq \Phi (\varepsilon ) \ast Ev,
\]
where $\Phi(x, t)\coloneqq (4\pi t)^{-3/2} \mathrm{e}^{-|x|^2/4t} $ denotes the heat kernel.   We note that $v_\varepsilon \in Y(R)$ for any $R>0$, since $v_\varepsilon$ is entire.   Moreover, since the heat flow is continuous in $H^r (\R^3)$ (see Appendix~6.5.1 in \cite{OP}, for example), we have
\eqnb\label{heat_cty}
\|  v_\varepsilon - v \|_{H^r(\R^3 )} \to 0
\eqne
as $\varepsilon \to 0$. Let $S_\varepsilon v \coloneqq w$ be the unique $H^r$ solution to the Stokes system
\eqnb\label{stokes}
\begin{split} -\Delta w + \nabla p &= -\Delta v_\varepsilon ,\\
\div\, w &=0\qquad \text{ in }\Omega,\\
w\cdot \nn &=0 ,\\
w\times \nn &= v_\varepsilon \times \nn \qquad \text{ on }\p \Omega
,
\end{split}
\eqne
see~\cite[Proposition~2.2 in Ch.~1]{T2}. Then,
\[
\| w \|_{H^r (\Omega )} \lec \| v_\varepsilon \|_{H^r (\Omega )} \leq \| v_\varepsilon \|_{H^r (\R^3 )} \leq \| E v \|_{H^r (\R^3 )} \lec \| v \|_{H^r (\Omega )},
\]
which shows that  $S_\varepsilon v \in V $   for every $\varepsilon >0$. Moreover,
\[
\| S_\varepsilon v - v \|_{H^r} \leq \| S_\varepsilon v - v_\varepsilon  \|_{H^r} +\|  v_\varepsilon - v \|_{H^r(\Omega )} \to 0
\]
as $\varepsilon \to 0$, where we used \eqref{heat_cty} and, for the first term, we noted that $\widetilde{w}\coloneqq S_\varepsilon v - v_\varepsilon $ is a solution to
\eqnb\label{stokes1}
\begin{split} -\Delta \widetilde{w} + \nabla p &= 0 ,\\
\div\, \widetilde{w} &=-\div \,v_{\varepsilon }   \qquad \text{ in }\Omega,\\
\widetilde{w}\cdot \nn &=-v_\varepsilon \cdot \nn ,\\
\widetilde{w}\times \nn &= 0 \qquad \text{ on }\p \Omega,
\end{split}
\eqne
and so, by \cite[Proposition~2.2 in Ch.~1]{T2} and \eqref{heat_cty},  
\[\| \widetilde{w} \|_{H^r (\Omega )} \lec \|\div\, v_{\varepsilon}  \|_{H^{r-1} (\Omega )}+\| v_{\varepsilon} \cdot \nn   \|_{H^{r-1/2} (\p \Omega )}  \to \|\div\, v  \|_{H^{r} (\Omega )} +\| v \cdot \nn   \|_{H^{r-1/2} (\p \Omega )}=0\]
as $\varepsilon \to 0$.

Furthermore,  since $ v_\varepsilon  $ is entire, we also have  $S_\varepsilon  v  \in \oY (1 )$ provided $\overline{\epsilon}/ \epsilon$ and $\epsilon$ in \eqref{EQ19c} are chosen sufficiently small (depending on $\Omega$), which follows from the analyticity of solutions to the Stokes problem \eqref{stokes}, see, for example \cite[Theorem~3.1 and the line above (5.3)]{KX} as well as \cite[Theorem~2.7]{CKV} for a similar property for time-dependent Stokes equations. See also   \cite{JKL1,JKL2} and \cite{KP2} for further background related to analyticity results.
\end{proof}

We now consider a sequence $\varepsilon_k \to 0$ and use Lemma~\ref{L_an_approx} to obtain a sequence of approximations
\[
u_0^{(k)} \coloneqq S_{\varepsilon_k } u_0 .
\]
For each $k$, we use Step 1 to obtain a solution $u^{(k)}$ to the Euler equations \eqref{euler} on $[0, T  ]$, with the initial condition~$u_0^{(k)}$. In particular, for each $k$, the choice of $\tau$ in the analytic spaces \eqref{EQ19c} is defined by \eqref{tau_choice_actual}, except with $\tau_0$ replaced by~$1$. In particular, we see that $H^r$  estimate \eqref{002b} remains uniform with respect to $k $, and so we can find a subsequence of $u^{(k)}$ which converges weakly-$*$ in $L^\infty ((0, T  );H^r)$. We also note that $\p_t u^{(k)}$ is bounded in $C([0, T  ];H^{r-1})$, uniformly with respect to $k$, which lets us use the Aubin-Lions lemma (see~\cite[Theorem~2.1 in Section~3.2]{T2}, for example) to get a solution $u \in C ([0, T  ];H^r)$ to \eqref{euler} on $[0, T  ]$ with initial data~$u_0$. The uniqueness of such solution can be obtained by a simple energy argument in $C([0, T  ];L^2)$. \\

\noindent\texttt{Step 3.} We construct a Sobolev solution on a Sobolev domain.\\
(Namely we prove Theorem~\ref{T1} in full generality.)\\

\noindent\texttt{Step~3a.} For a given $\varepsilon>0$, we find an analytic domain  $\Omega^\varepsilon$  that approximates~$\Omega$.\\

Namely, since  $\Omega $ is a $H^{r+2}$ domain, the signed distance function $\phi (x)$ (with the convention that  $\phi(x)>0$ for $x\in \Omega$) is a Lipschitz function such that $H^{r+2} (\widetilde{\Omega })$ and $|\nabla \phi |$ is bounded below by a positive constant on $\widetilde{\Omega}$, where $\widetilde{\Omega}$ is some open neighbourhood of $\p \Omega$. This can be proved by a Sobolev version of the classical fact \cite[Lemma~14.16]{GT}, but for convenience of the reader, we include a detailed proof in Appendix~\ref{sec_level}. 

 Given $\varepsilon>0$, we set \[\phi_\varepsilon \coloneqq \Phi (\varepsilon )\ast  d  ,
\]
where $\Phi $ denotes the 3D heat kernel, and we let 
\[
 \Omega^\varepsilon  \coloneqq \{ \phi_\varepsilon >0 \}.
\]
Then, $\phi_\varepsilon $ is analytic by properties of the heat kernel (see \cite{KP2}, for example), and so $ \Omega^\varepsilon $ is an analytic domain; recall~\eqref{def_an_domain}. Moreover, $ \Omega^\varepsilon $ can be described locally as a graph of an analytic function, which can be proven using an analytic version of the Implicit Function Theorem; see~\cite{KP1}, and $ \Omega^\varepsilon  \to \Omega  $ in $H^{r+2}$ as $\varepsilon \to 0$, by following the analysis of \cite[Section~6.3]{A}. \\

\noindent\texttt{Step~3b.} For each $\varepsilon>0$ we construct  a measure-preserving $H^{r+2}$ diffeomorphism $\eta \colon \Omega \to  \Omega^\varepsilon $  such that $\| \eta - \mathrm{id} \|_{H^{r+2}}\to 0$ as $\varepsilon \to 0$.\\

 To this end, we will first construct a diffeomorphism $\eta \colon \Omega \to  \Omega^\varepsilon $ and then apply a correction to make it measure-preserving.\\

We let $r>0$ and $B_1, \ldots , B_L$ be a collection of open balls of radius $r$ that cover $\p \Omega $ and are such that, for each  $l\in \{ 1,\ldots , L\}$, the surface $\p \Omega\cap B_l$ is (up to rotation) the graph of an $H^{r+2}$ function.  Clearly, for sufficiently small $\varepsilon >0$, the same is true for $\p  \Omega^\varepsilon \cap B_l$ for $l\in \{ 1, \ldots , L\}$. We let $\beta \in (0,1)$ be sufficiently small so that  the collection $\{ (1-\beta ) B_l \}_{l=1}^L$ also covers $\p \Omega$, where $(1-\beta )B_l$ denotes the ball centered at the same point as $B_l$ and $(1-\beta)$ times larger radius. 
We let $\xi \in C_0^\infty (B(r))$ be such that $\xi = 1$ on $B((1-\beta )r)$, where $B(a)$ denotes an open ball of radius $a>0$ centered at $0$. We set $\eta_0 \coloneqq \mathrm{id}$. For each $l\in \{ 1, \ldots , L\}$ we define $Q_l \colon \R^3 \to \R^3$  as  a composition of the translation such that $B_l$ is mapped onto $B(r)$ and a rotation that maps $\eta_{l-1} (\p \Omega ) \cap B_l $, $\eta_{l-1} (\p  \Omega^\varepsilon ) \cap B_l$ onto graphs of some $H^{r+2}$ functions $F_0 $, $F$, respectively. Namely, $F_0,F\in H^{r+2} (\R^2 ; \R )$ are such that 
\[
Q_l (\eta_{l-1} (\p \Omega ) \cap B_l ) = \{ x\colon x_3 = F_0 (x') \} \cap B(r),
\]
and similarly for $F$, where we set $x'\coloneqq (x_1,x_2)$ for brevity. Letting
\[
{G_l} (x) \coloneqq x + (F(x')-F_0(x')) \xi (x) e_3,
\]
we see that $G_l$ is a $H^{r+2}$ bijection of $B(0,r)$ onto itself, provided $\varepsilon $ is sufficiently small. Moreover, $G_l = \mathrm{id}$ in a neighbourhood of $\p B(r)$. We set
\eqnb\label{eta_def}
\eta_l \coloneqq Q_l^{-1} \circ G_l \circ Q_l \circ \eta_{l-1}.
\eqne
We note that, for sufficiently small $\varepsilon >0$, the gradient of the second term on the right-hand side of \eqref{eta_def} is negligible compared to the identity matrix $I$, arising from the first term, guaranteeing that $\eta_l$ is a diffeomorphism with inverse of the regularity dictated by $F$, $F_0$ (using the Implicit Function Theorem), namely $H^{r+2}$, which is the same as~$\eta_l$. Moreover, for small $\varepsilon >0$, the collection $\{ (1-\beta ) B_l \}_{l=1}^L $ covers $\eta_l (\p  \Omega^\varepsilon )$ and the surface $\eta_l (\p \Omega^\varepsilon  ) \cap B_l $ is, up to rotation, the graph of an $H^{r+2}$ function. Hence, we can inductively define all $\eta_l$'s, until $\eta \coloneqq \eta_L$. By construction, $\eta $  satisfies the claim, and an inductive argument shows that $\eta \in H^{r+2}$ and that
\eqnb\label{eta_approx}
\| \eta - \id \|_{H^{r+2}} \to 0\qquad \text{ as }\varepsilon \to 0,
\eqne
as required.\\

 We now apply a correction to $\eta $ to make it measure-preserving. The main condition for this to be possible is the preservation of total measure. For this we set
\eqnb\label{def_lambda}
\lambda \coloneqq |\Omega |^{-\frac{1}3} \left(\int_\Omega ( \det \nabla \eta  )^{-1}\right)^{1/3},
\eqne
and we note that \eqref{eta_approx} gives 
\eqnb\label{lambda_conv}
\lambda\to 1 \qquad \text{ as } \varepsilon \to 0.
\eqne
Clearly, $\lambda  \Omega^\varepsilon  = \lambda \eta (\Omega )$ is an analytic domain as it is a homothethic transformation of an analytic domain. We now let $\zeta \in H^{r+2} (\Omega )$ be a diffeomorphism of $\Omega$ onto itself such that
\eqnb\label{zeta_eq}
\begin{split}
\det \nabla \zeta (x) &= \left( \det \nabla (\lambda \eta (x)) \right)^{-1}=: f(x)  \hspace{0.8cm}\text{ for }x\in \Omega ,\\
 \zeta (x)  & = x\hspace{4.5cm} \text{ for }x\in \p \Omega ,
\end{split}
\eqne
 and $\| \zeta - \mathrm{id} \|_{H^{r+2} } \leq C (\Omega  ) \| f - 1 \|_{H^{r+1}}  $. The existence of such $\xi$ follows from the generalization of the classical result of Dacorogna and Moser \cite{DM} to the Sobolev setting, see \cite[Lemma~9]{Ye} for details. In particular note that the total volume condition, which is the solvability condition of \eqref{zeta_eq}, holds, since $\int_\Omega f = | \Omega |$ by the definition \eqref{def_lambda} of $\lambda$.\\
 
 The point of \eqref{zeta_eq} is that 
\[
\eta'  \coloneqq \lambda \eta \circ \zeta 
\] 
is the desired measure-preserving correction of $\eta$. Indeed, it is a diffeomorphism of $\Omega $ onto $\lambda  \Omega^\varepsilon $, 
\eqnb\label{eta'_Hr_conv}
\| \nabla \eta'  -\mathrm{id} \|_{H^{r+2}} \to 0
\eqne
as $\varepsilon \to 0$, 
due to the same property of $\eta$, \eqref{lambda_conv} and the inequality following  \eqref{zeta_eq}.
Moreover,
\[
\det \nabla \eta'  = \det \nabla (\lambda \eta ) \,\, \det \nabla \zeta =1, 
\]
showing that $\eta'$ is measure-preserving. For simplicity, we now redefine $\eta$ to denote the correction $\eta'$ we have just constructed, and we redefine  $ \Omega^\varepsilon \coloneqq \eta (\Omega )$. Note that \eqref{eta'_Hr_conv} also gives that 
\eqnb\label{a_approx}
\| a - I \|_{H^{r+1}} \to 0\qquad \text{ as }\varepsilon \to 0,
\eqne
where $I$ denotes the identity matrix and $a\coloneqq (\nabla \eta )^{-1}$ (i.e., $a_{ji} = \p_i (\eta^{-1})_j$). \vspace{0.3cm}\\

\noindent\texttt{Step~3c.} For each sufficiently small $\varepsilon >0$, we construct divergence-free $v_0\in H^r ( \Omega^\varepsilon )$ such that $v_0 \cdot N =0$ on $\p  \Omega^\varepsilon $, where $N$ denotes the unit normal vector to $\p  \Omega^\varepsilon $ and $\| u_0 - v_0 \circ \eta \|_{H^r (\Omega )} \to 0$ as $\varepsilon \to 0$.\\

To this end, for each $\varepsilon >0$, we let $U \in H^r (\Omega)$ be such that 
\eqnb\label{U0_system}
\begin{cases}
\div\, U = \div ((I-a )U) ,&\\
\curl\, U =\curl \, u_0  &\text{ in }\Omega,\\
U \cdot \nn = ((I - a )U) \cdot \nn  &\text{ on }\p \Omega,
\end{cases} 
\eqne
where we used the summation convention, and $a\coloneqq (\nabla \eta )^{-1}$ (i.e., $a_{ji} = \p_i (\eta^{-1})_j$).  Given $U$, we  claim that
\[
 v \coloneqq U\circ \eta^{-1} \in H^r ( \Omega^\varepsilon )
\]
satisfies the claim of this step. Indeed, $v$  satisfies
\[
\begin{split}
\div\, v &= \p_i v_i = \p_j  U_i \p_i (\eta^{-1})_j = \p_j  U_i a_{ji}= \div (aU) =0  
\end{split}
\]
in $ \Omega^\varepsilon $,  where, in the fourth equality, we used the Piola identity $\p_j a_{ji} =0$, which is a consequence of the measure-preserving property $\det \, \nabla \eta =1$, established in the previous step, and we also used  the first equation of \eqref{U0_system} in the last step. Moreover $N(y)=  a(x)^T \nn(x) / |a^T \nn|$, where $y=\eta (x)$. Thus, since
\[
v(y)\cdot (a(\eta^{-1}(y))^T \nn(\eta^{-1}(y))) = U \cdot (a^T \nn ) = U_i a_{ji} \nn_j = (aU ) \cdot \nn = ((a-I )U ) \cdot \nn + U\cdot \nn =0
\]
for every $y\in \p  \Omega^\varepsilon $, where we omitted  $x=\eta^{-1}(y)$  in the notation, we have $v\cdot N=0$ on $\p  \Omega^\varepsilon $, as required.\\

 We now comment on the existence of $U$ solving \eqref{U0_system}. We let $\mathcal{V} \coloneqq H^{r-1} (\Omega ) \times H^{r-1} (\Omega ) \times H^{r-1/2} (\p \Omega )$, and we denote by $\mathcal{V}_0$ the  subspace of $\mathcal{V}$ consisting of triples $(f,g,b)$ such that
\eqnb\label{div_cond_S}
\int_\Omega f = \int_{\p \Omega } b
\eqne
and
\eqnb\label{23_cond}
\int_{\gamma } g \cdot \nn =0 
\eqne
for every connected component $\gamma$ of $\p \Omega$.
By \cite[Theorem~1.1]{CS} there exists a mapping $S \colon \mathcal{V}_0 \to H^r (\Omega )$ such that $w\coloneqq S(f,g,b)$ solves 
\eqnb\label{w_system}
\begin{cases}
\div\, w = f ,&\\
\curl\, w =g &\text{ in }\Omega,\\
w \cdot \nn = b &\text{ on }\p \Omega
\end{cases} 
\eqne
and satisfies the estimate
\eqnb\label{w_est}
\| w \|_{H^r} \lec \| f \|_{H^{r-1}} + \| g \|_{H^{r-1}} + \| b \|_{H^{r-1/2} (\p \Omega )}.
\eqne
Using the mapping $S$ we see that $U$ solves  \eqref{U0_system}  if
\eqnb\label{U_sys_new}
U = S\Bigl(\div ((I-a) U ) ,0 , ((I-a) U)\cdot \nn \Bigr) + S(0,\curl\, u_0 ,0) \qquad \text{ in } H^r (\Omega ).
\eqne
However, it is important to keep in mind that  $S$ is not necessarily a bounded, linear operator, since  solutions to the system  \eqref{w_system} are not unique, unless $\Omega$ is assumed to be simply connected (or a disjoint union of simply connected sets, see \cite{CS} for details). This means that we cannot use a fixed-point formulation to obtain a solution $U$ to \eqref{U_sys_new}. Instead we proceed in a more  constructive manner: We let $A_1 \coloneqq u_0 $, and, for $n\geq 2$, we set
\eqnb\label{369}
A_n \coloneqq S \left( \div \left( (I-a) A_{n-1} \right), 
  0 ,
  ((I-a)A_{n-1} )\cdot \nn
                \right).
\eqne
(Note that the conditions \eqref{div_cond_S}--\eqref{23_cond} are trivially satisfied.) We also set $U_n \coloneqq \sum_{k=1}^n A_k$. This way $U_n$ solves 
\eqnb\label{eq_Vn}
\begin{cases}
\div\, U_n = (\delta_{ij}-a_{ij})\partial_{i} (U_{n-1})_j
,&\\
\curl\, U_n =\curl \, u_0  &\text{ in }\Omega,\\
U_n \cdot \nn = ((I - a )U_{n-1}) \cdot \nn
, &\text{ on }\p \Omega
\end{cases} 
\eqne
for $n\geq 2$. In other words, $\{ A_n \}_{n\geq 2}$ is a sequence of ``additions'' to the first approximation $U_1=A_1 = u_0$ which give us a solution $U$ to \eqref{U0_system} in the limit, which we now verify. Note that, by \eqref{w_est} and \eqref{369}, 
\[\| A_n \|_{H^r} \leq C  \| (I-a)A_{n-1}  \|_{H^r} \lec C^{n-1} \| I-a \|_{H^r}^{n-1} \| A_1\|_{H^r}\]
for all $n\geq 1$, where $C>0$ is a constant depending only on $\Omega$ and the implicit constant in \eqref{w_est}. Thus, for  $\varepsilon >0$  sufficiently small so that $\| I-a \|_{H^r} \leq 1/2C$, we have
\[
\sum_{k\geq 1} \| A_k \|_{H^r} \lec \| u_0 \|_{H^r} \sum_{k\geq 1} (C\| I-a \|_{H^r} )^{k-1} \leq 2 \| u_0 \|_{H^r} <\infty.
\]
Therefore, there exists $U\in H^r$ such that $\| U_n -  U\|_{H^r} \to 0 $ as $n\to \infty$, and taking the limit $n\to \infty$ in \eqref{eq_Vn} gives \eqref{U0_system}, as required.\\

Finally, the last claim of this step follows by noting that $U - u_0 = \sum_{k\geq 2} A_k$, so that
\[
\| U - u_0 \|_{H^r} \leq \sum_{k\geq 2} \| A_k \|_{H^r} \leq  \| u_0 \|_{H^r} \sum_{k\geq 2} (C\| I-a \|_{H^r} )^{k-1} \leq 2 C \| u_0 \|_{H^r} \| I- a \|_{H^r} \to 0
\]
as $\varepsilon \to 0$. \vspace{0.3cm}\\

\noindent\texttt{Step~3d.} We take the limit $\varepsilon \to 0$.\\

For all sufficiently small $\varepsilon >0$ we consider an analytic domain $ \Omega^\varepsilon $ given by Step 3a, we obtain a diffeomorphism $\eta \colon \Omega \to  \Omega^\varepsilon $ given by Step 3b, and we consider $v_0\in H^r ( \Omega^\varepsilon )$ provided by Step 3c. We denote by $v$ the solution to the Euler equation \eqref{euler} on $ \Omega^\varepsilon $ given by Step 2. We note that
\eqnb\label{apr_v}
\| v (t) \|_{H^r ( \Omega^\varepsilon )} \leq  \| v_0 \|_{H^r ( \Omega^\varepsilon )} + C_0 ( \Omega^\varepsilon  ,r ) \int_0^t \| v(s) \|^2_{H^r ( \Omega^\varepsilon  )} \d s
\eqne
 for $t\in [0,T_0]$, where $T_0\coloneqq 1/4 C_0 ( \Omega^\varepsilon ,r) \| v_0 \|_{H^{r}( \Omega^\varepsilon )}$, by the Sobolev estimate~\eqref{sob_est}. We verify in Step~3e below that 
\eqnb\label{ts_step3e}
\text{ there exists }C_0 =C_0 (\Omega ,r) \text{ such that }
C_0( \Omega^\varepsilon  , r) \leq C_0 (\Omega ,r )
\eqne
for all sufficiently small $\varepsilon$, which, together with the fact that $\| v_0 \|_{H^r(\Omega^\varepsilon )} \to \| u_0 \|_{H^r(\Omega )}$ (a consequence of Steps~3b--c), gives that $T_0 \geq T$ for sufficiently small $\varepsilon$. (Recall~\eqref{time_lb} that, given $C_0$, we have $T=1/8C_0  \| u_0 \|_{H^r(\Omega )}$.)

 Together with Step 3c, the inequality~\eqref{apr_v} gives that, for sufficiently small $\varepsilon >0$,
\[
\| v(t) \circ \eta \|_{H^r (\Omega )} \leq 2 \| v_0 \|_{H^r (\Omega^\varepsilon )}\leq  4 \| u_0 \|_{H^r (\Omega )} \qquad \text{ for } t\in [0, T  ].
\]

Similarly as in Step 2, we also obtain a bound on $\p_t v (t) $ in~$H^{r-1} ( \Omega^\varepsilon )$. In fact, mapping back to $\Omega$ we obtain that $\p_t v \circ \eta $ is bounded in $C([0, T  ]; H^{r-1})$, uniformly with respect to~$\varepsilon$. This allows us to extract a sequence $\varepsilon_k\to 0$ along which $v\circ \eta $ (which depends on $\varepsilon_k$) converges  in $C([0, T  ];H^{r} (\Omega ))$  to some $u\in C([0, T  ];H^{r} (\Omega ))$.  In order to verify that $u$ solves the Euler equation \eqref{euler}, we first observe that, since $v$ satisfies the Euler equation in $ \Omega^\varepsilon $, it also satisfies the weak formulation 
\eqnb\label{v_weak}
 \int_\Omega \left( v(\cdot ,t)-v(\cdot ,s) \right) \cdot  \phi  = - \int_s^t \int_\Omega  v_j  v_i   \p_k \phi_i 
\eqne
for all $s,t\in [0,T]$ and every $\phi \in V$. Letting $w\coloneqq v\circ \eta$,  we see that \eqref{v_weak} implies that  
\eqnb\label{w_sys1}
 \int_\Omega \left( w(\cdot ,t)-w(\cdot ,s) \right) \cdot  \varphi  = - \int_s^t \int_\Omega  w_j  w_i   \p_k \varphi_i  a_{kj}
\eqne
for all $s,t \in [0,T]$ and every $\phi \in V $, where  $\varphi \in H^r (\Omega )$ is a solution to
\eqnb\label{varphi_sys}
\begin{cases}
\div\, \varphi = \div ((I-a )\varphi ) ,&\\
\curl\, \varphi =\curl \, \phi  &\text{ in }\Omega,\\
\varphi \cdot \nn = ((I - a )\varphi) \cdot \nn , &\text{ on }\p \Omega
\end{cases} 
\eqne
such that $\| \varphi - \phi \|_{H^r} \to 0$ as $\varepsilon \to 0$. (The existence of such $\varphi$ follows in the same way as in the case of system~\eqref{U0_system}, discussed in Step~3c above.) The point of~\eqref{varphi_sys} is that $\varphi \circ \eta^{-1} \in H^r ( \Omega^\varepsilon )$ is divergence-free and satisfies the non-penetration boundary condition on $\p  \Omega^\varepsilon $, so that it could be used as a test function to obtain a weak formulation of the Euler equation for $v$ in $ \Omega^\varepsilon $. The formulation is~\eqref{w_sys1} after a change of variable. Taking the limit $\varepsilon \to 0$ in~\eqref{w_sys1} we thus obtain that $ \int_\Omega \left( u(\cdot ,t)-u(\cdot ,s) \right) \cdot  \phi  = - \int_s^t \int_\Omega  u_j  u_i   \p_j \phi_i  $ for all $s,t\in [0,T]$, $\phi \in V$, which shows that $u$ satisfies the weak formulation of the Euler equation on $\Omega$, as required.    \\

Finally,   taking the limit $\varepsilon \to 0$ in \eqref{apr_v}, we see that $u$ is the required solution satisfying the a~priori bound \eqref{sob_est1}.\vspace{0.3cm}\\

 \noindent\texttt{Step~3e.} We prove \eqref{ts_step3e}.\\

We let $u(x,t) \coloneqq v(\eta(x),t)$, and we will show that
\eqnb\label{u_apr}
\| u(t) \|_{H^r} \leq \| u(s) \|_{H^r} + C_0 \int_s^t \|  u \|_{H^r }^2
\eqne
for all sufficiently small $\varepsilon$, where  $C_0>0$ is a constant independent of $\varepsilon$. The claim then follows from the fact that $\| \eta - \mathrm{id }\|_{H^{r+2}}\to 0$ as $\varepsilon \to 0$ (recall Step~3b).\\

We first note that, since $\p_t v + (v\cdot \nabla )v + \nabla p=0$ in $\Omega^\varepsilon$, the pressure function $p$ satisfies
\eqnb\label{p_sys0}
\begin{split}
-\Delta p &= \p_i v_j \p_j v_i\hspace{1cm} \text{ in }\Omega^\varepsilon,\\
\p_{N} p & = -(v_i \p_i v ) \cdot {N }\hspace{0.6cm} \text{ on }\p \Omega^\varepsilon, 
\end{split}
\eqne
where $N(y)$ denotes the outward unit normal vector to $\p \Omega^{\varepsilon }$ at $y$. We claim that \eqref{p_sys0} can be rewritten as
\eqnb\label{p_rewrit}
\begin{split}
-\Delta p & = \p_i v_j \p_j v_i \hspace{2cm} \text{ in }\Omega^\varepsilon ,\\
\p_i p \p_i \phi_\varepsilon &= - v_i v_i \p_{ij} \phi_\varepsilon \hspace{2cm} \text{ on }\p \Omega^\varepsilon ,
\end{split}
\eqne
where $\phi_\varepsilon$ is the level set function obtained by convolving $\phi$ with the heat kernel (recall Step~3a). This claim can be justified using a  derivative reduction trick due to Temam~\cite{T1}, which we now discuss.

Since $\Omega^\varepsilon = \{ \phi_\varepsilon = 0 \}$ (recall Step 3a) and $\nabla \phi_\varepsilon = -| \nabla \phi_\varepsilon |  \, N$ on $\p \Omega^\varepsilon$, we see that the boundary condition $v\cdot N=0$ on $\p \Omega^\varepsilon$ implies that 
\[
v\cdot \nabla \phi_\varepsilon =0 \quad \text{ on } \p \Omega^\varepsilon .
\]
Thus $\p \Omega^\varepsilon$ is also the $0$-level set of $v\cdot \nabla \phi_\varepsilon$ (apart from being the $0$-level set of $\phi_\varepsilon$), and so 
\eqnb\label{levelsets}\nabla ( v\cdot \nabla \phi_\varepsilon )= h\, \nabla \phi_\varepsilon ,
\eqne
where $h$ is some scalar function. Therefore, using again the fact that $N=- \nabla \phi_\varepsilon / |\nabla \phi_\varepsilon |$, we have
\[
(v_i \p_i v)\cdot N =-\frac1{|\nabla \phi_\varepsilon |} v_i \p_i v_j \p_j \phi_\varepsilon = - \frac1{|\nabla \phi_\varepsilon |} v_i  v_{j} \p_{ij} \phi_\varepsilon,  
\]
where we used \eqref{levelsets} in the last step. Thus \eqref{p_rewrit} follows, since $\nabla p \cdot N = -\p_i p \p_i \phi_\varepsilon / |\nabla \phi_\varepsilon |$.\\

We now set $q (x) \coloneqq p (\eta (x))$, $u(x) \coloneqq v(\eta (x))$, so that  
\[
p(y)= q(\eta^{-1} (y)),\qquad v(y) =  u(\eta^{-1} (y))
\]
for every $y\in \Omega^\varepsilon$, and we aim to rewrite \eqref{p_rewrit} as an elliptic boundary value problem for $q$. We have that 
\[
\p_i p  = \p_j q (\eta^{-1} ) \p_i(\eta^{-1})_j = a_{ji} \p_j q,  
\]
and so
\[
\Delta p = \p_k (a_{ki} a_{ji} \p_j q ),
\]
where we used the Piola identity $\p_k a_{ki}=0$. Similarly, we obtain $\p_i v_j = \p_k u_j  a_{ki}$ for all $i,j \in \{ 1,2,3\}$, so that the Poisson equation for $p$ from \eqref{p_rewrit} becomes the Poisson equation
\eqnb\label{q_eq1}
- \Delta q = \p_k ( ( a_{ki} a_{ji} - \delta_{kj} ) \p_j q ) + \p_k u_j  \p_m u_i a_{ki} a_{mj}
\eqne
for $q$ on $\Omega$. Similarly, the boundary condition from \eqref{p_rewrit} becomes
\eqnb\label{bc_new}
\p_k q a_{ki} \p_l (\phi_\varepsilon \circ \eta ) a_{li} = -u_i u_j \p_k \left( a_{ki} a_{lj} \p_l (\phi_\varepsilon \circ \eta ) \right) ,
\eqne
so that the Neumann boundary condition for $q$ on $\p \Omega$ can be written as 
\eqnb\label{q_eq2}\begin{split}
-\p_{\mathsf{n}} q &= -\p_k q \, \mathsf{n}_k = \p_k q \p_k \phi /|\nabla \phi |  \\
& = \frac{\p_k q}{|\nabla \phi |}  \p_l \phi a_{ki} a_{li} +\frac{\p_k q}{|\nabla \phi |} ( \delta_{kl} - a_{ki}a_{li}) \p_l \phi   \\
&= -\frac{1}{|\nabla \phi |} u_i u_j \p_k \left( a_{ki} a_{lj} \p_l (\phi_\varepsilon \circ \eta ) \right) +\frac{\p_k q}{|\nabla \phi |}  a_{ki}a_{li} ( \p_l \phi - \p_l (\phi_\varepsilon \circ \eta )) + \frac{\p_k q}{|\nabla \phi |} ( \delta_{kl} - a_{ki}a_{li}) \p_l \phi  ,
\end{split}
\eqne
where we noted that $\mathsf{n} = -\nabla \phi / |\nabla \phi |$ on $\p \Omega$ in the second equality, and used \eqref{bc_new} in the last equality.

Having rewritten \eqref{p_rewrit} into the boundary value problem \eqref{q_eq1}, \eqref{q_eq2}, we now recall that $| \nabla \phi |^{-1} \in H^r (\widetilde{\Omega} )$ for some neighborhood $\widetilde{\Omega}$ of $\p \Omega$, and that spaces $H^r$ and $H^{r+1}$ are algebras. Thus, elliptic regularity of the boundary value problem \eqref{q_eq1}, \eqref{q_eq2} for $q$ gives 
\[
\begin{split}
\| q \|_{H^{r+1}} &\lec_r \| u \|_{H^r}^2\left( 1   +\|  \phi_\varepsilon \circ \eta  \|_{H^{r+2} ( \widehat{\Omega} )} \right) \\
&\quad + \| q \|_{H^{r+1}} \left(  \| I -a^T a \|_{H^r} + \| \phi - (\phi_\varepsilon \circ \eta ) \|_{H^{r+1}(\widehat{\Omega})} + \| I - a^T a \|_{H^r } \| \phi \|_{H^{r+1}(\widetilde{\Omega})} \right) \\
&\lec_{\Omega,r}  \| u \|_{H^r}^2 + \left( \| u \|_{H^r}^2 +  \| q \|_{H^{r+1}} \right) \left( \| \phi - (\phi_\varepsilon \circ \eta ) \|_{H^{r+2}( \widehat{\Omega} )} + \| I -  a \|_{H^r }  \right) 
\end{split}
\]
where $\widehat{\Omega} \subset \subset \widetilde{\Omega } $ is a neighbourhood of $\p \Omega$, and we used the  fact that $\| a \|_{H^{r+1}} \lec 1$ (a consequence of \eqref{a_approx}) in the first inequality. In the second inequality we incorporated the $\phi$ dependence into the implicit constant (dependent on $\Omega$). Hence, since $\phi_\varepsilon \to \phi$ in $H^{r+2}$ (a property of convolution with heat kernel) and $\eta \to \mathrm{id}$ in $H^{r+2}$ (recall~\eqref{eta_approx}), the last bracket in the above inequality vanishes as $\varepsilon \to 0$, yielding  
\eqnb\label{q_eqf}
\| q \|_{H^{r+1}} \lec_{\Omega ,r } \| u \|_{H^r}^2.
\eqne  

This estimate on $q$ allows us to perform Sobolev estimates on $u$. Indeed, since $\p_t v + v_j \p_j v + \nabla p =0$ in $\Omega^\varepsilon$ we have 
\eqnb\label{u_pde}
\p_t u_i  + u_j \p_k u_i a_{kj}+ \p_m q \,a_{mi} =0
\eqne
in $\Omega$, so that, taking any partial derivative $D^\alpha$, with $|\alpha | =r$, of \eqref{u_pde}, multiplying by $D^\alpha u_i$, summing in $i$ and integrating over $\Omega$, we obtain
\eqnb\label{u_Hr}
\frac12 \frac{\d }{\d t }\| D^\alpha u \| \lec_r \sum_{\substack{\beta \leq \alpha \\ \beta >0 }} \| D^\beta ( a u) \cdot \nabla D^{\alpha-\beta } u \| + \| a^T \nabla q \|_{H^r} \lec_r \| u \|_{H^r}^2 ,
\eqne
where in the first inequality we used the fact that $\div \, v =0 \Leftrightarrow \div (au) =0$ (by the Piola identity) to obtain the cancellation
\[\begin{split}
\int_\Omega a_{kj} u_j \p_k D^\alpha u_i \, D^\alpha u_i& = \frac12  \int_\Omega (au ) \cdot  \nabla |D^\alpha u |^2 = \int_{\p \Omega } | D^\alpha u |^2 (au)\cdot \mathsf{n} = \int_{\p \Omega } | D^\alpha u |^2 u\cdot (a^T \mathsf{n}) \\
&= \int_{\p \Omega^\varepsilon } | (D^\alpha u)\circ \eta^{-1} |^2 \underbrace{v \cdot N}_{=0} \, |(a^T \mathsf{n} )\circ \eta^{-1}| =0 
\end{split}
\]
(recall from Step 3c that $N(y) = a^T \mathsf{n}(x) / |a^T \mathsf{n}(x)|$); we also used \eqref{q_eqf} and $ \| a\|_{H^r}\lec 1$ in the second inequality. Supplementing \eqref{u_Hr} with an $L^2$ estimate, we can integrate it in time to obtain \eqref{u_apr}, as required.

\section*{Acknowledgments}
IK~was supported in part by the
NSF grant
DMS-2205493, while
WO~was supported the NSF grant no.~DMS-2511556 and by the Simons grant SFI-MPS-TSM-00014233. 



\appendix
\section{The signed distance function of a Sobolev domain}\label{sec_level}

As mentioned in Step 3a above, here we show that, if $d(x)$ denotes the signed distance function of $\p \Omega$ of a $H^{r+2}$ domain $\Omega$ (with $\Omega = \{ d >0 \}$), then $d$ is Lipschitz and there exists an open neighhourhood $\widetilde{\Omega }$ of $\p \Omega$, such that $d\in H^{r+2} (\widetilde{\Omega })$ and $| \nabla d |$ is bounded away from $0$ on $\widetilde{\Omega}$. To this end, we first recall \cite[(14.91)]{GT} that $d$ has  Lipschitz constant $1$. Given $x\in \R^n$, let $y\in \R^{n-1}$ be such that $(y,f(y))$ is the closest point to $x$ on the graph of $f$. Note that there exists a sufficiently small open neighbourhood of $\p \Omega$ such that $y(x)$ is uniquely defined if $x$ belongs to it.   We will show that for each $x_0\in \p \Omega $ there exists $\delta >0$ such that 
\eqnb\label{d_cl_loc}
d\in H^{r+2} (B(x_0 ,\delta ))\quad \text{ and }\quad | \nabla d | =1 \text{ in }B(x_0,\delta ).
\eqne
The claim then follows by compactness of $\p \Omega$.\\

In order to prove \eqref{d_cl_loc}, we assume, without loss of generality,  that $x_0=0$, $f(0)=0$ and $\nabla f (0 ) =0$. We set 
\eqnb\label{Hdelta} 
H(\delta ) \coloneqq \| \nabla f \|_{L^\infty ( B'(2\delta ) )} \lec \delta, \eqne
where we used the notation $B'(2\delta )$ for the open ball  in $\R^{n-1}$ which is centered at $0$ and has radius $2\delta $. 
Note that, since $y$ is the minimizer of the distance function, $(x_1-y_1)^2 + \ldots + (x_{n-1}- y_{n-1} )^2 + (x_n - f(y))^2$, it satisfies 
\eqnb\label{eq_y}
y = x' + \nabla f (y) (x_n - f(y)).
\eqne
We set $y^0 (x) \coloneqq x'$, and
\eqnb\label{eq_yk}
y^k  \coloneqq x' + \nabla f (y^{k-1}) (x_n - f(y^{k-1})) 
\eqne
for $k\geq 1$.   \vspace{0.5cm}\\

We first show that $y$ is characterized as the pointwise limit of the $y^k$'s, that is
\eqnb\label{c1}
y(x) = \lim_{k\to \infty } y^k (x)\in B' (2\delta ) \qquad \text{ for all } x\in B(\delta ).
\eqne
Indeed, note that
\[
| y^1 - x' | = | \nabla f (x' ) (x_n - f(x' ))| \leq C_f \delta ( \delta + |f(x' ) |) \leq  C_f \delta^2  
\]
for each $x\in B(\delta )$ if $\delta $ is sufficiently small,
where we used \eqref{Hdelta} in the first and second inequalities. Moreover, for each $k\geq 1$, $x\in B(\delta )$, we have
\eqnb\label{ind_k}
|y^{k+1} - y^{k}| \leq C_f \delta | y^{k} - y^{k-1} | ,
\eqne
provided $\delta $ is sufficiently small. Indeed, given $l\geq 1$, if \eqref{ind_k} holds for $k\leq  l-1$, we have 
\eqnb\label{yl_bd}| y^{l-1} | \leq |x' | + \sum_{k=1}^{l-1} | y^{k} - y^{k-1} |  \leq \delta + \delta \sum_{k=1}^{l-1} (C_f \delta)^k \leq 2\delta  \eqne
if $\delta $ is sufficiently small. In particular, \eqref{Hdelta} gives $\| f \circ y^{l-1} \|_{L^\infty (B(\delta ))} \leq C_f \delta^2 $. Thus, writing
\[
\begin{split}
y^{k+1} - y^k  &= \nabla f(y^k ) (x_n - f (y^k )) - \nabla f(y^{k-1} ) (x_n - f (y^{k-1} ))\\
&=  \nabla f(y^k ) (f(y^{k-1} ) - f(y^k ) ) + \left( \nabla f(y^k ) - \nabla f (y^{k-1}) \right) (x_n - f(y^{k-1}) ),
\end{split}
\]
we obtain
\[
| y^{k+1} - y^k | \leq  \left( \| \nabla f \|_{\infty } + \| D^2 f \|_\infty \| x_n - f\circ y^{k-1} \|_\infty  \right) |y^{k }  - y^{k-1} | \leq C_f \delta |y^{k }  - y^{k-1} |,
\]
proving \eqref{ind_k} for $k=l$. Therefore, if $\delta <C_f^{-1}$, \eqref{ind_k} gives that $\{ y^k (x) \}_{k\geq 1}$ is Cauchy in $\R^{n-1}$ for each $x\in B(\delta )$. This and \eqref{yl_bd} show that $y^k \to y$ for some $y\in B' (2\delta )$. Taking the limit $k\to \infty $ in \eqref{eq_yk} shows \eqref{c1}, as required.\vspace{0.5cm}\\

We now want to show that 
\eqnb\label{c2}
\sup_{k\geq 1 } \| y^k \|_{H^{r+1}(B(\delta ))} <\infty  ,
\eqne
provided $\delta $ is sufficiently small. To this end, we first recall that, for every $\delta \in (0,1/2)$ and every $q\colon \R^{n} \supset B(\delta ) \to B'(2\delta ) \subset \R^{n-1}$ satisfying
\eqnb
\notag
\|  q (x)- x'  \|_{C^1 (\overline{B(\delta )}) } \leq \frac1{2\sqrt{n}},
\eqne
  we have
\eqnb\label{coar}
\int_{B(\delta ) }  \left| g(q(x),x_n ) \right|^2  \d x \lec \int_{B'(2\delta ) \times (-\delta , \delta ) }  \left| g(x ) \right|^2  \d x    
\eqne
for any $g\in L^2 (B'(2\delta ) \times (-\delta , \delta ) )$. This inequality can be easily proved using the Coarea Formula, but here we give a brief elementary proof. We see that $Q(x)\coloneqq (q(x),x_n)$ satisfies
\[
|Q(x)-Q(z) | \geq | x-z| - |(\mathrm{id}' - q)(x) - (\mathrm{id }' -q )(z) | \geq \frac12 |x-z|
\]
for all $x,z\in B(\delta )$, which gives that $Q$ is a $C^1$ injection of $B(\delta ) $ onto $Q(B(\delta )) \subset B'(2\delta ) \times (-\delta ,\delta )$. We have also denoted by $\mathrm{id}'$ the identity map on $\R^{n-1}$. Moreover, 
\eqnb\label{det0}
|\det \nabla Q |\geq 1/2^n >0 
\eqne
in $B(\delta )$. Indeed, if, for any $x\in B(\delta )$, $\lambda \in \C$ is an eigenvalue of $\nabla Q (x) $, with some eigenvector $v\in \C^{n-1}$ then $\lambda - 1 $ is an eigenvalue of $\nabla Q (x) - I$ so that
\[
| (\lambda -1 ) v | = | (\nabla Q(x) - I ) v | \leq \frac12 | v | .
\]
Thus $| \lambda | \geq 1 - |\lambda - 1 | \geq 1/2$, and \eqref{det0} follows. In particular $|\det (\nabla Q^{-1})| \lec_n 1$, so that  $Q$ is a $C^1 $ diffeomorphism of $B(\delta )$, and so
\[\begin{split}
\int_{B(\delta ) }  \left| g \circ Q  \right|^2   &
= \int_{Q(B(\delta ))} |g |^2 | \det (\nabla Q^{-1})|    \lec_n  \int_{B'(2\delta ) \times (-\delta , \delta ) } |g |^2  ,  
\end{split}
\]
showing \eqref{coar}.\\

Thus, before proving \eqref{c2}, we first verify that 
\eqnb\label{C1_bd}
\| y^k (x) - x' \|_{C^1 (\overline{B(\delta )})} \leq \frac1{2\sqrt{n}} 
\eqne
for all $k\geq 0$. Note that the case $k=0$ is trivial. We set 
\[
h(x ) \coloneqq \nabla f (x' ) (x_n - f(x')), 
\]
and note that, by \eqref{eq_yk}, 
\[
y^k(x) - x' = h(y^{k-1}(x),x_n)
\]
for all $x\in B(\delta )$, $k\geq 1$. Moreover,
\[
\| h \|_{C^1 (\overline{B'(2\delta )\times (-\delta ,\delta )})} \to 0
\]
as $\delta \to 0$. Thus, supposing that \eqref{C1_bd} holds for $k-1$, we have 
\[
\| y^k(x) - x'  \|_{C^1 (\overline{B(\delta )})} \lec \| h \|_{C^1 (\overline{B'(2\delta )\times (-\delta ,\delta )})} \underbrace{\left( 1+\| y^{k-1} \|_{C^1 (\overline{B(\delta )})} \right)}_{\lec 1},
\]
which vanishes as $\delta \to 0$. Taking $\delta $ sufficiently small gives \eqref{C1_bd}.\\

Having observed \eqref{coar} and \eqref{C1_bd}, we now show that 
\eqnb\label{sob_bd}
 \| y^{k}(x) - x' \|_{H^{r+1} (B(\delta ))} \leq \frac1{2} \eqne
for all $k\geq 0$. Indeed, it is trivial for $k=0$. For  $k\geq 1$ we suppose that \eqref{sob_bd} holds for $1,\ldots , k-1$, and observe that 
\[\begin{split}
\| y^{k}(x) - x' \|_{H^{r+1} (B(\delta )) } &\leq 
\| h( y^{k-1}(x),x) \|_{H^{r+1} (B(\delta )) } \lec_n \| h \|_{H^{r+1} (B'(2\delta ) \times (-\delta , \delta ))} (1+ \| y^{k-1} \|_{H^{r+1}} )^{r+1}\\ &
\lec_{r,n} \delta^{1/2} \| f \|_{H^{r+2} (B'(2 \delta ))} (1+ \| f \|_{H^{r+1} (B'(2 \delta ))} ) ,
\end{split}\]
where in the second step we used \eqref{C1_bd} to apply \eqref{coar} with $h(x) \coloneqq D^\alpha g$ for all multiindices $\alpha$ with $|\alpha | \leq r+1$. We also used the inductive assumption \eqref{sob_bd} in the last step. Taking sufficiently small $\delta$ proves \eqref{sob_bd}. The claim \eqref{c2} now follows from \eqref{sob_bd} and the triangle inequality.\vspace{0.5cm}\\

We now observe that \eqref{c1} and \eqref{c2} imply that  \[y\in H^{r+1}(B(\delta ))\] for sufficiently small $\delta $.\\

Having shown that $y\in H^{r+1} (B(\delta ))$, we denote by
\[
\gamma (x) \coloneqq (y(x) , f(y(x)))\in \p \Omega 
\]
the closest point to $x$ on $\p \Omega$, and we  denote by 
\eqnb
\notag
\mathsf{n} (\gamma (x)) \coloneqq \frac{1}{\left( 1+ | \nabla f (y(x)) |^2 \right)^{\frac12} } \left( \p_1 f (y(x)),\ldots , \p_{n-1} f(y(x)),-1 \right) 
\eqne
the outward normal vector to $\p \Omega$ at $\gamma (x)$. Since $y\in H^{r+1}$ and $f\in H^{r+2}$ we see that $\gamma \in H^{r+1} (B(\delta ))$ and $\mathsf{n} \circ \gamma   \in H^{r+1} (B(\delta ))$. 

Let us now focus on the distance function $d(x)$. We can write any $x\in B(\delta )$ in the form
\[
x= \gamma (x) - \mathsf{n} (\gamma (x)) d(x).
\]
Taking the dot product with $\mathsf{n} (\gamma (x))$ gives 
\[
d(x) = (\gamma (x)-x) \cdot \mathsf{n} (\gamma (x)).
\]

We note that at this point the regularities of $\mathsf{n} \circ \gamma $ and $\gamma $ suggest that we should only expect that $d\in H^{r+1}$. Remarkably, this can be lifted to $r+2$ by noting that
\eqnb\label{whoisnad}
\nabla d (x) = -\mathsf{n} (\gamma (x))
\eqne
for $x\in B(\delta )$. Indeed, we have 
\eqnb\label{d_char}
d^2(x) = | \gamma (x) - x |^2 = |y(x) - x'|^2 + |f(y(x))-x_n |^2,
\eqne
so that
\eqnb\label{nad_char}
\begin{split}
\nabla d^2&= 2\begin{pmatrix}
(y_1 - x_1 )(\p_1 y_1 -1) +  (y_2-x_2) \p_1 y_2 + \ldots +  (y_{n-1} - x_{n-1})\p_1 y_{n-1} + (f-x_n  )\p_k f \p_1y_k \\
\vdots \\
(y_1 - x_1 )\p_{n-1} y_1  +  (y_2-x_2) \p_{n-1} y_2 + \ldots +  (y_{n-1} - x_{n-1})(\p_{n-1} y_{n-1}-1) + (f-x_n  )\p_k f \p_{n-1}y_k\\
(y_1 - x_1 )\p_n y_1 +  (y_2-x_2) \p_n y_2 + \ldots +  (y_{n-1} - x_{n-1})\p_n y_{n-1} + (f-x_n  )(\p_k f \p_n y_k -1)
\end{pmatrix}\\
&=  2(f-x_n) \begin{pmatrix}
-\p_1 f (\p_1 y_1 -1) -\p_2 f   \p_1 y_2 - \ldots -\p_{n-1} f \p_1 y_{n-1} + \p_k f \p_1y_k \\
\vdots \\
-\p_1 f \p_{n-1} y_1  -\p_2 f \p_{n-1} y_2 - \ldots - \p_{n-1} f (\p_{n-1} y_{n-1}-1) + \p_k f \p_{n-1}y_k\\
-\p_1 f \p_n y_1 - \p_2 f  \p_n y_2 - \ldots -  \p_{n-1} f \p_n y_{n-1} + (\p_k f \p_n y_k -1)
\end{pmatrix}\\
&= 2(f-x_n) \begin{pmatrix}
\nabla f \\
-1 
\end{pmatrix},
\end{split}
\eqne
where  all components of $\nabla d^2$ and $\nabla y$ are evaluated at $x$, while $f$ and all components of $\nabla f$ are evaluated at $y(x)$. Here,  in the second equality, we recalled from \eqref{eq_y} that $y_i-x_i  =-  \p_i f  ( f-x_n  )$ for all $i=1,\ldots , n-1$. Using \eqref{eq_y} also in \eqref{d_char} we see that
\[
d = (f-x_n )\sqrt{ |\nabla f |^2 +1 }.
\]
This and \eqref{nad_char} give
\[
\nabla d = \left( |\nabla f|^2 +1 \right)^{-\frac12}  \begin{pmatrix}
\nabla f \\
-1  
\end{pmatrix}= -\mathsf{n} \circ \gamma ,
\]
yielding \eqref{whoisnad}, as required. Clearly \eqref{whoisnad} gives a gain of one derivative, so that $d\in H^{r+2}$, and also shows that $| \nabla d |=1 $ in $B(\delta )$, which concludes the proof of  \eqref{d_cl_loc}.

\end{document}